\documentclass[12pt]{amsart}
\usepackage{verbatim,amscd,amssymb}
\usepackage{cref}

\usepackage{graphicx}

\usepackage{tikz}
\usetikzlibrary{matrix}

\newtheorem{thm}{Theorem}[section]

\newtheorem{lem}[thm]{Lemma}

\newtheorem{cor}[thm]{Corollary}

\newtheorem*{thmS}{Straightening Theorem}
\newtheorem*{thma}{Theorem A}
\newtheorem*{thmb}{Theorem B}

\theoremstyle{definition}
\newtheorem{defn}[thm]{Definition}

\newcommand{\uc}{\mathbb{S}}
\newcommand{\Ac}{\mathcal{A}}
\newcommand{\al}{\alpha}
\newcommand{\toh}{\mathrm{TH}}
\newcommand{\imp}{\mathrm{Imp}}
\newcommand{\be}{\beta}
\newcommand{\ga}{\gamma}
\newcommand{\ta}{\theta}
\newcommand{\0}{\emptyset}
\newcommand{\sm}{\setminus}
\newcommand{\vp}{\varphi}
\newcommand{\ol}{\overline}
\newcommand{\D}{\mathbb{D}}
\newcommand{\iD}{\D^\infty}
\newcommand{\C}{\mathbb{C}}
\newcommand{\si}{\sigma}
\newcommand{\e}{\varepsilon}

\newcommand{\disk}{\mathbb{D}}
\newcommand{\diam}{\mathrm{diam}}

\begin{document}

\title[Locally Connected Models]{Non-degenerate locally connected models
for plane continua and Julia sets}
\author{}
\date{May 9, 2016}
\author[A.~Blokh]{Alexander~Blokh}

\author[L.~Oversteegen]{Lex Oversteegen}

\author[V.~Timorin]{Vladlen~Timorin}

\address[Alexander~Blokh and Lex~Oversteegen]
{Department of Mathematics\\ University of Alabama at Birmingham\\
Birmingham, AL 35294-1170}

\address[Vladlen~Timorin]
{Faculty of Mathematics\\
Laboratory of Algebraic Geometry and its Applications\\
Higher School of Economics\\
Vavilova St. 7, 112312 Moscow, Russia }

\address[Vladlen~Timorin]
{Independent University of Moscow\\
Bolshoy Vlasyevskiy Pereulok 11, 119002 Moscow, Russia}

\email[Alexander~Blokh]{ablokh@math.uab.edu}
\email[Lex~Oversteegen]{overstee@math.uab.edu}
\email[Vladlen~Timorin]{vtimorin@hse.ru}



\begin{abstract}
Suppose that a $X$ is an \emph{unshielded} plane continuum (i.e., $X$
coincides with the boundary of the unbounded complementary component of
$X$). Then there exists a \emph{finest monotone} map $m:X\to L$, where
$L$ is a locally connected continuum (i.e., $m^{-1}(y)$ is connected
for each $y\in L$, and any monotone map $\vp:X\to L'$ onto a locally
connected continuum is a composition $\vp=\vp'\circ m$ where $\vp':L\to
L'$ is monotone). Such finest locally connected model $L$ of $X$ is
easier to understand because $L$ is locally connected (in particular it
can be described by a picture) and represents the finest but still
understandable decomposition of $X$ into possibly complicated but
pairwise disjoint \emph{fibers} (point-preimages) of $m$. However, in
some cases (i.e., in  case $X$ is indecomposable) $L$ is a singleton. In this
paper we provide sufficient conditions for the existence of a
non-degenerate model depending on the existence of certain subcontinua
of $X$ and apply these results to the connected Julia sets of
polynomials.
\end{abstract}

\maketitle

\section{Introduction}

A natural approach to studying a topological space $X$ is to model
$X$ using simpler and easier to deal with spaces. By this we mean
finding a factor space of $X$ such that both the quotient map
$m:X\to L$ and the model space $L$ are manageable. In this paper we
consider only \emph{plane continua}; in that setting we view
\emph{monotone} maps and \emph{locally connected continua} as
manageable. This leads to the concept of the \emph{finest locally
connected model under a monotone map} of a plane continuum
$X$.

The concept was inspired by Jan Kiwi who approached the problem of
modeling from the point of view of (complex) dynamical systems. To
state Kiwi's results we need a few definitions.

\begin{defn}[Semiconjugacy of maps]\label{d:semi}
Two maps $f:X\to X$ and $g:Y\to Y$ are said to be \emph{semiconjugate}
if there exists a map $\psi:X\to Y$ such that $\psi\circ f=g\circ
\psi$. In other words, the following diagram is commutative:


\begin{center}

\begin{tikzpicture}
  \matrix (m) [matrix of math nodes,row sep=3em,column sep=4em,minimum width=2em]
  {
     X & X \\
     Y & Y \\};
  \path[-stealth]
    (m-1-1) edge node [left] {$\psi$} (m-2-1)
            edge node [below] {$f$} (m-1-2)
    (m-2-1.east|-m-2-2) edge node [below] {$g$}
            (m-2-2)
    (m-1-2) edge node [right] {$\psi$} (m-2-2);
\end{tikzpicture}

\end{center}


\end{defn}




We also need to define a concept of a \emph{monotone map}.

\begin{defn}[Monotone map]\label{d:monotone}
A map $f:X\to Y$ is \emph{monotone} provided for each $y\in Y$,
$f^{-1}(y)$ is connected. \end{defn}

In what follows let $\C$ be the complex plane and let $\widehat \C$ be the complex sphere.
In his paper \cite{kiw04} Kiwi proves that if a polynomial $P$ with
connected Julia set $J(P)$ has no periodic points with multipliers
which are complex numbers of modulus $1$ and irrational argument then
$P$ can be semiconjugate to a so-called \emph{topological polynomial}
$f_P:\C\to \C$. The semiconjugacy $\vp:\C\to \C$ is a monotone map
which is one-to-one outside the Julia set $J(P)$; thus, basically $\vp$
collapses some subcontinua of $J(P)$ (\emph{fibers} of $\vp$) to
points. The topological polynomial $f_P$ is a branched covering map
such that $\vp(J(P))$ is a locally connected continuum with
well-understood structure and dynamics described by so-called
\emph{laminations}.

As mentioned above, Kiwi's approach to the problem was based upon
dynamical systems' considerations. Later on in \cite{bco11} it was
discovered that an approach based upon continuum theory yields results
that extend those of \cite{kiw04} while also being applicable in a
purely topological setting. We need a few definitions.

\begin{defn}\label{d:finest}
Let $X$ be a continuum. A continuum $Y$ is a \emph{finest locally
connected model for $X$} if there exists a monotone map $m:X\to Y$ so
that for any monotone map $f:X\to Z$, where $Z$ is a locally connected
continuum, there exists a monotone map $g:Y\to Z$ so that $g\circ m=f$;
then we will call the map $m$ a \emph{finest monotone map}.
\end{defn}

We consider this notion on the plane in the context of so-called \emph{unshielded} continua.

\begin{defn}\label{d:unshield}
Given a compact set $X$ in the plane, let $U_X^\infty$  denote the
unbounded complementary domain of $X$. The set $\toh(X)=\C\sm
U_X^\infty$ is called the \emph{topological hull} of $X$. A compact set
$X$ in the plane is \emph{unshielded} provided $X$ coincides with the
boundary $\partial(U_X^\infty)$ of the unbounded complementary domain
$U_X^\infty$ of $X$. Observe that any subcontinuum of an unshielded
continuum is unshielded.
\end{defn}

The following theorem shows that a finest locally connected model and
a finest monotone map are well-defined for unshielded plane continua (in
\cite{bco13} the result was extended to plane compacta).

\begin{thm}[\cite{bco11}]\label{t:models}
Every unshielded plane continuum $X$ has a finest locally connected
model $Y$ and a finest monotone map $m$. Moreover, any two finest
locally connected models of an unshielded continuum $X$ are
homeomorphic. Furthermore, $m$ can be extended to a monotone map $\widehat \C\to
\widehat \C$ which maps $\infty$ to $\infty$, in $\C\sm X$ collapses
only those complementary domains to $X$ whose boundaries are collapsed
by $m$, and is a homeomorphism elsewhere in $\widehat \C\sm X$.
\end{thm}

By \cref{t:models} we can talk about \emph{the} finest locally
connected model of an unshielded continuum and \emph{the} finest
monotone map. It follows that if an unshielded plane continuum $X$ has
the finest locally connected model which is non-degenerate then its
topological hull $\toh(X)$ also has a non-degenerate model.

In particular, the connected Julia set of a polynomial admits a finest
locally connected model. However, in some cases the finest locally
connected model is a single point; in this case we say that the finest
locally connected model is \emph{degenerate} while otherwise we call
such model \emph{non-degenerate}. Obviously, if the finest model is
degenerate, all information regarding the continuum $X$ is lost while
otherwise some of the structure of $X$ is preserved in its model. This
shows the importance of the fact that the finest locally connected
model of an unshielded continuum is non-degenerate. In the present
paper we will study conditions under which the finest locally connected
model is non-degenerate. Moreover, in the final section we apply this
result to polynomial dynamics.

\section{Statement of main results and applications}

In this section we assume knowledge of basic concepts of continuum
theory and complex dynamics (all necessary definitions are given in
detail later in the sections of the paper containing the proofs of our
main results). Denote the open unit disk by $\D$ and the disk at infinity (i.e.,
$\mathbb{C}\sm \ol{\D}$) by $\iD$. We will identify the unit circle
$\uc=\partial \D=\partial \iD$ with $\mathbb R/\mathbb Z$ and call
the induced order on $\uc$ the \emph{circular order}. Note that the
circular order is not defined for a pair of points in $\uc$, but if
$x,y,z\in \uc$ are three distinct points, then $x<y<z$ in the
circular order if, when traveling from $x$ in the positive direction
along $\uc$, we encounter $y$ before $z$. Thus, from now on a single
point $x\in \uc$ will be denoted by the corresponding angle, i.e. by
a number $\al\in [0,1)$ with  $x=e^{2\pi\alpha i}$.

If $X$ is a plane continuum, then by the Riemann mapping theorem there
exists a conformal map $\psi_X:\iD\to U_X^\infty$ with derivative
converging to a real number as $|z|\to\infty$. \emph{External rays} of
$X$ foliate $U_X^\infty$ and serve as a major tool in studying the topology
of $\partial(U_X^\infty)$.

\begin{defn}[External rays]\label{d:exteray}
Let $X$ be a plane continuum. By an \emph{external ray of $X$} we mean
the image of the radial line segment with argument $2\pi\al$ under the
Riemann map $\psi_X$; in what follows, this image will be denoted by $R_X(\al)$.
In other words,
\[R_X(\alpha)=\psi_X(\{r\,e^{2\pi \alpha i}\mid r>1\}).\]
If we do not want to emphasize the argument we denote an external ray
of $X$ by $R_X$. We say that the external ray $R_X(\al)$ \emph{lands on $x_\al\in X$} provided
$\ol{R_X(\al)}\sm R_X(\al)=x_\al$.
\end{defn}

We will mostly consider external rays for unshielded plane continua $X$
(in that case $X=\partial(U_X^\infty)$) such as connected Julia sets of
complex polynomials, however sometimes we work with external rays of
other plane continua (e.g., we consider external rays of connected
\emph{filled} Julia sets). Observe that the unbounded complementary
domain $U^\infty_X$ of a continuum $X$ coincides with the unbounded
complementary domain $U^\infty_{\toh(X)}$ of its topological hull.
Therefore we can (and will) interchangeably talk about external rays of
$X$ and/or external rays of $\toh(X)$.

\begin{defn}[Strategically placed subcontinua]\label{d:strapla}
Suppose that $Y$ is a subcontinuum of an unshielded continuum $X$ in the complex
plane. Then we say that $Y$ is \emph{strategically placed in $X$}
provided that there exists a dense set $\Ac(Y, X)=\Ac\subset \uc$ so that:
\begin{enumerate}
\item for each $\al\in \Ac$, $R_Y(\al)$ lands on a point $y_\al\in Y$,
\item the set of points $\{y_\al\}_{\al\in\Ac}$ is dense in $Y$,
\item there exists a circle order preserving function $p:\Ac\to\uc$ so that for each $\al\in\Ac$
the ray $R_X(p(\al))$ lands on $y_\al$.
\end{enumerate}
In this case we say that $\Ac$ is an \emph{anchor set (of $Y$)} and $p:\Ac\to\uc$ is an \emph{external connecting function (of $Y$)}.
\end{defn}

Since $p$ preserves order, $p$ is one-to-one but  we do not assume that
$p$ is continuous.

Theorem A is our main continuum theory result. It shows that in some cases
the fact that a subcontinuum has a non-degenerate finest locally connected model implies
that the same can be said about the continuum itself.

\begin{thma}
Let $X$ be an unshielded plane continuum. If  $Y$ is
strategically placed in $X$, and $Y$ has a non-degenerate finest locally
connected model, then $X$ has a non-degenerate finest locally connected model.
\end{thma}

The main applications of this result are in complex dynamics. Namely,
the following theorem holds.

\begin{thmb}
Suppose that $P:\C\to \C$ is a polynomial and  $J^*\subset J(P)$  is a
continuum which is a polynomial-like Julia set of $P^n$
for some $n>0$. If $J^*$ has a non-degenerate finest locally connected
model, then so does $J(P)$.
\end{thmb}

In Subsection~\ref{ss:non-conn} we rely upon\, \cite{bco13} and prove a
version of Theorem B for disconnected Julia sets.

\section{Proof of Theorem A}

In the first subsection of this section we give various standard definitions.
 Then we prove Theorem A.

\subsection{Basic definitions}

The notion of the \emph{principal set} is used in studying the limit
behavior external rays.

\begin{defn}[Principal set]\label{d:princi}
Given an external ray $R_X(\al)$ of an unshielded continuum $X$ we denote by
$\Pr_X(\al)$ the set $\ol{R_X(\al)}\setminus R_X(\al)$ and call it the
\emph{principal set of the ray $R_X(\al)$.} If $\Pr_X(\al)$ is a single
point $y$ we say that the external ray $R_X(\al)$ \emph{lands on $y$.}

More generally, let $T\subset U_X^\infty$ be an image of $\mathbb
R_+=(0, \infty)$ under a continuous map $\psi:\mathbb R_+\to \C$ such
that $\lim_{t\to \infty}|\psi(t)|=\infty$ while $\0\ne \ol{T}\setminus
T\subset X$. Then we say that $T=\psi(\mathbb R_+)$ \emph{accumulates}
in $X$, denote $\ol{T}\setminus T$ by $\Pr_X(T)$ and call it the
\emph{principal set of the curve $T$ which accumulates in $X$.} If
$\Pr_X(T)$ is a single point $y$ we say that the curve $T$ which
accumulates in $X$ \emph{lands on $y$.}
\end{defn}

Another important definition is that of a \emph{crosscut} (see, e.g.,
\cite{mil06} for details).

\begin{defn}[Crosscuts]\label{d:crosscut}
A \emph{crosscut $C$} of $X$ is an open  arc $C\subset U_X^\infty$ so
that its closure is a closed arc with two distinct endpoints both of
which belong to $X$. A \emph{fundamental chain $\{C_i\}$ (of
crosscuts)}  is a sequence of crosscuts $C_i$ of $X$  such that the
following holds:

\begin{enumerate}

\item $\ol{C_i}\cap \ol{C_j}=\0$ if $i\ne j$,

\item for each $i$, $C_i$ separates $C_{i+1}$ from infinity in
    $U_X^\infty$, and

\item $\lim \diam (C_i)=0$.

\end{enumerate}

For each crosscut $C$ of $X$ its \emph{shadow $S_C$} is the closure of
the bounded complementary domain of $\mathbb C\sm [X\cup C]$ whose
boundary contains $C$.
\end{defn}

Note that every fundamental chain $\{C_i\}$ corresponds to a unique
point $\al\in\uc$ defined by $\lim (\varphi_X)^{-1}(C_i)=\al$ and in
this case we say that $\{C_i\}$ is a \emph{fundamental chain for
$\al$}.

\begin{defn}[Impressions]\label{d:impress}
The \emph{($X$-)impression} $\imp_X(\al)$ is defined as
\[\imp_X(\al)=\bigcap S_{C_i}, \text{ where $\{C_i\}$ is a fundamental chain for } \al.\]
\end{defn}

It is easy to see that both $\Pr_X(\al)$ and $\imp_X(\al)$ are
continua, that $\Pr_X(\al)\subset \imp_X(\al)$ and that $\imp_X(\al)$
is independent of the choice of the fundamental chain for $\al$.
Moreover, let $X$ be an unshielded continuum. Then, even though
$\bigcup_\al \Pr_X(\al)$ can be a proper subset of the continuum $X$,
$\bigcup_\al \imp_X(\al)=X$.

\subsection{Proof of Theorem A}
Let us recall that the notion of a subcontinuum $Y$ strategically
placed in an unshielded continuum $X$ was introduced in
\cref{d:strapla}. A part of this definition is a function $p$
(so-called \emph{external connecting function}) of a dense set
$\Ac\subset \uc$ to $\uc$ which preserves circle order and maps angles
such that for each $\al\in \Ac$, both the ray $R_Y(\al)$ and the ray
$R_X(p(\al))$ land on a point $y_\al\in Y$. We will show below that the
choice of the function $p$ is severely restricted. Moreover, the
condition in Lemma~\ref{homotope} characterizes the situation in which
a subcontinuum is strategically placed in an unshielded continuum (and so this
characteristic can be used as an alternative definition of the fact
that $Y$ is strategically placed in $X$).

\begin{lem}\label{homotope}
Suppose that $Y\subset X$ are unshielded planar continua. 
Then the following are equivalent:
\begin{enumerate}
\item $Y$
is strategically placed in $X$ with anchor set $\Ac$,
\item \label{ht}  There exists a dense set $\Ac\subset \uc$ so that
    for $\al\in\mathcal A$, there exists $\be(\al)=\be\in\uc$ so that
    the ray $R_X(\be)$ also lands on $y_\al$ and the rays $R_Y(\al)$
    and $R_X(\be)$ are homotopic in $\{y_\al\}\cup \C\sm Y$ under a
    homotopy which fixes the landing point $y_\al\in Y$.
\end{enumerate}
\end{lem}

\begin{proof} Suppose that $Y$
is strategically placed in $X$ with anchor set $\Ac$
and $p:\Ac\to \uc $ as
the external connecting function.
Suppose that $\al\in \Ac$ and  the ray $R_Y(\al)$ lands on $y_\al\in Y$. Clearly,
$R_X(p(\al))$ can be viewed as a curve in $\C\sm Y$ which
accumulates in $Y$; more precisely, we can say that $R_X(p(\al))$
lands on $y_\al$. Thus, $R_X(p(\al))$ is homotopic to some external
ray $R_Y(\be)$ in $\C\sm Y$ under a homotopy which fixes $y_\al$ (so that
the ray $R_Y(\be)$ lands on $y_\al$ too). If
$\al\ne \be$, then both components of $\C\sm \ol{R_Y(\al)\cup
R_Y(\be)}$ intersect $Y$ (because two \emph{distinct} external rays
of $Y$ which land on the same point of $Y$ cannot be homotopic
outside $Y$).

Choose $\ga_1,\ga_2\in\Ac$ so that $\al<\ga_1<\be<\ga_2<\al$ and
$R_Y(\ga_1)$ and $R_Y(\ga_2)$ land in different components $C_1, C_2$
of $\C\sm \ol{R_Y(\al)\cup R_Y(\be)}$, respectively. Let $R_Y(\ga_1)$ land
on a point $y_{\ga_1}\in C_1\cap Y$ and let $R_Y(\ga_2)$ land on a point
$y_{\ga_2}$ in $C_2\cap Y$. Then $R_X(p(\ga_2))$ is an external ray of $X$
which also lands on $y_{\ga_2}$.  Since $\al<\ga_1<\ga_2$, then
$p(\al)<p(\ga_1)<p(\ga_2)$.

Consider the set $E$ of angles in $p(\Ac)$ which belong to $(p(\al),
p(\ga_2))$. Consider also the component $\widetilde E$ of the set
$$\C\sm [R_X(p(\al))\cup \{y_\al\} \cup R_X(p(\ga_2)\cup (C_2\cap Y)]$$
containing external rays of $X$ with arguments from $E$. The
external rays of $X$ with argument in $E$ can only land on points
from the boundary of $\widetilde E$ but not on points from other external
rays of $X$; thus, the external rays of $X$ with argument in $E$
can only land on points from $[C_2\cap
Y]\cup \{y_\al\}$. In particular this must be true for the ray
$R_X(p(\ga_1))$. However by definition this ray must land on the point
$y_{\ga_1}\in C_1\cap Y$, a contradiction.

Suppose next that condition~(\ref{ht}) holds. It suffices to show that
the map $p:\al\to \be(\al)$ preserves circular order. Recall that by
$\psi_Y:\disk^\infty \to U^\infty_Y$ we denote the conformal map with
derivative converging to a real number as $|z|\to\infty$. Similarly,
let $\psi_X:\disk^\infty \to U^\infty_X$ be the corresponding Riemann
map from the complement of the closed unit disk to the unbounded
component of $X$. Assume that $\al_1<\al_2<\al_3\in \uc$. Let
$\be_j=\be(\al_j)$; then the rays $\psi_Y^{-1}(R_X(\be_j))=\widehat R_j$
land on the points $e^{2\pi \be_j\, i}=x_j\in \uc$ and $x_1<x_2<x_3$.
Let $S_r$ be the circle $\psi_Y^{-1}\circ \psi_X(\{z\in \mathbb C \mid \,
|z|=r\})$ with an induced circular order $<$. As $r\searrow 1$, $S_r$
intersects $\widehat R_j$ in a unique point $y_j(r)$ and
$\lim_{r\searrow 1} y_j(r)=x_j$. This implies that
$y_1(r)<y_2(r)<y_3(r)$ as required.
\end{proof}


Suppose that $Y\subset X$ is strategically placed in $X$ with anchor
set $\Ac$. Then \cref{homotope} implies that for any $\al\in \Ac\subset
\uc$ the ray $\psi^{-1}_Y(R_X(p(\al)))$ lands on $\al\in \uc$. This
visualization is useful in the proof of the next lemma that describes
intersections between closures of components of $X\sm Y$ and $Y$. It
follows easily from the assumptions that $X$ is unshielded and $Y$ is
strategically placed in $X$.

\begin{lem}\label{compinimp}
Suppose that $X$ is an unshielded continuum and $Y\subset X$ is a continuum strategically
placed in $X$. If $C$ is a component of $X\sm Y$, then
\[|\ol{\psi^{-1}_Y(C)}\cap \partial \mathbb D|=1.\]
In particular, if $\al=\ol{\psi^{-1}_Y(C)}\cap \partial \mathbb D$, then
$\ol{C}\cap Y\subset \imp_Y(\al)$.
\end{lem}

\begin{proof}
Observe that $\psi^{-1}_Y(C)$ is a connected subset of $\C\sm \ol{\mathbb D}$ (because $X$ is unshielded). It follows that
if $\ol{\psi^{-1}_Y(C)}\cap \partial \mathbb D$  is non-degenerate then there exists a non-degenerate
arc $[p,q]\subset \partial \mathbb D$ such that \emph{any} (not necessarily radial) ray $T\in \C\sm \ol{\mathbb D}$
which lands at $\be\in (p, q)$ must intersect $\psi^{-1}_Y(C)$.
Choose $\al\in(p,q)\cap \Ac$. By
Lemma~\ref{homotope}, $\psi^{-1}_Y(R_X(p(\al))$ lands on $\al$.  Then
$\psi^{-1}_Y(C)\cap \psi^{-1}_Y(R_X(p(\al))\ne\0$ and, hence, $C\cap
R_X(p(\al))\ne\0$, a contradiction. To prove the last claim of the lemma choose
a fundamental system of crosscuts $B_i$ such that $\psi^{-1}_Y(B_i)$
converge to $\al=\ol{\psi^{-1}_Y(C)}\cap \partial \mathbb D$. Then by definition
their shadows converge to $\imp_Y(\al)$. Since all these shadows contain $\ol{C}\cap Y$,
it follows that $\ol{C}\cap Y\subset \imp_Y(\al)$ as desired.
\end{proof}

\cref{compinimp} motivates the following definition.

\begin{defn}[Angles associated with components of $X\sm Y$]\label{d:roots}
Suppose that $X$ is an unshielded continuum and $Y\subset X$ is a
continuum strategically placed in $X$. Given a component
$C$ of $X\sm Y$ we call the angle $\al$ such that
$\ol{\psi^{-1}_Y(C)}\cap \partial \mathbb D=\{\al\}$
the \emph{angle associated with $C$} and denote it by $\al(C)$ which defines a map from the
family of components of $X\sm Y$ to the unit circle. We also define
the function $C$ which associates to any point $x\in X\sm Y$ the component
$C(x)$ of $X\sm Y$ such that $x\in C(x)$. Finally, we consider a function
$\al:X\sm Y\to \uc$ defined as $\al(x)=\al(C(x))$ for every $x\in X\sm Y$.
\end{defn}

Using the terminology introduced in \cref{d:roots} we can restate
\cref{compinimp} as follows: if $X$ is an unshielded continuum and
$Y\subset X$ is a continuum strategically placed in $X$ then for every
component $C$ of $X\sm Y$ we have $\ol{C}\cap Y\subset \imp_Y(\al(C))$.

We will need the following geometric lemma.

\begin{lem}\label{l:converg}
Suppose that $X$ is an unshielded continuum and $Y\subset X$ is a
continuum strategically placed in $X$. Let $\{x_i\}$ be a sequence of
points of $X\sm Y$ such that $x_i\to x$ and $\al(x_i)\to \be$. Then
either $x\in \imp_Y(\be)$ or $x\in X\sm Y$ and $\al(x)=\be$. In
particular, the map $\al:X\sm Y\to \uc$ is continuous.
\end{lem}

\begin{proof}
Since impressions are upper semi-continuous and because
$\ol{C(x_i)}\cap Y\subset \imp_Y(\al(x_i))$ by \cref{compinimp}, we have that
$$\limsup \ol{C(x_i)}\cap Y\subset \limsup \imp_Y(\al(x_i))\subset \imp_Y(\be).$$

\noindent If angles $\ta, \ta', \ga', \ga\in \Ac$ are close to $\be$
and $\ta<\ta'<\be<\ga'<\ga$ then for sufficiently large $i$ we have
that $\al(x_i)=\al_i\in (\ta, \ga)$, and by \cref{homotope} all components
$C(x_i)$ are contained in the same appropriately chosen component
$Z(\ta, \ga)$ of $\C\sm Y\cup R_X(p(\ta))\cup R_X(p(\ga))$ containing
external rays of $X$ with arguments from $(p(\ta), p(\ga))$. Since the set
$$Q(\ta, \ga)=Z(\ta, \ga)\cup R_X(p(\ta))\cup R_X(p(\ga))\cup Y$$
is closed this implies that $x\in Q(\ta, \ga)$.
Consider now two possibilities.

\noindent 1. Suppose that $x\notin Y$ but $\al(x)\ne \be$. Then we can
choose angles $\ta$ and $\ga$ so that $\al(x)\notin [\ta, \ga]$ and
therefore $C(x)$ is disjoint from $Q(\ta, \ga)$, a contradiction with
the fact that $x\in Q(\ta, \ga)$. Thus, if $x\notin Y$ then $\al(x)=\be$.

\noindent 2. Suppose that $x\in Y$. Let us show that then $x\in
\imp_Y(\be)$. Indeed, choose angles $\ta, \ta', \ga', \ga$ as above.
Draw crosscuts $T(\ta, \ga)=T$ and then $T(\ta', \ga')=T'$ inside the
shadow $S_T$ of $T$. Then for some $\e>0$ every point $z\notin S_T$ of
a component $C$ of $X\sm Y$ with $\al(C)\in (\ta', \ga')$ is at least
$\e$-distant from $Y$. In particular, if $x_i\notin S_T$ then the
distance between $x_i$ and $Y$ is at least $\e$. Since $x_i\to x\in Y$, it follows
that $x_i\in S_T$ for sufficiently large $i$, and hence that $x\in \imp_Y(\be)$.

This completes the proof of the lemma.
\end{proof}

Recall, that given a map $h:A\to B$ we call point-inverses of $h$
\emph{($h$-)fibers.} The following lemma is proven in \cite{bco11}.

\begin{lem}[\cite{bco11}]\label{l:fib-in-imp}
Let $K\subset \C$ be an unshielded continuum and $m:K\to Z$ be a monotone map
of $K$ to a locally connected continuum $Z$. Then all fibers of $m$ are unions
of $K$-impressions (equivalently, $m$ collapses any $K$-impression to a point).
In particular, this holds for the finest monotone map $m_K$ of $K$.
\end{lem}

We are ready to prove Theorem A.

\begin{thma}
Let $X$ be an unshielded plane continuum. If  $Y$ is
strategically placed in $X$, and $Y$ has a non-degenerate finest locally
connected model, then $X$ has a non-degenerate finest locally connected model.
\end{thma}


\begin{proof}
By \cref{t:models} it suffices to show that there exists a monotone map
from $X$ to a non-degenerate locally connected continuum $L$. Since $Y$
has a non-degenerate finest locally connected model, then there exists
the finest monotone map $m_Y:Y\to L$ so that $L$ is a non-degenerate
locally connected continuum.  We will extend the map $m_Y$ to a
monotone map  $m:X\to L$ as follows: for every $x\in X\sm Y$ set
$m(x)=m_Y(\imp_Y(\al(x)))$. Observe that since by
Lemma~\ref{l:fib-in-imp} the map $m_Y$ collapses all $Y$-impressions to
points, then the map $m(x)$ is well-defined. Let us show that this map
has the desired properties.

First we show that $m$ is continuous. To see that, we first show that
if $x_i\to x$ then one can find a subsequence $x_{i_j}$ such that
$m(x_{i_j})\to m(x)$. This is obvious if infinitely many points $x_i$
belong to $Y$ because $m_Y$ is continuous. Thus we may assume that
$x_i\in X\sm Y$ for every $i$. Choose a subsequence $x_{i_j}$ so that
$\al(x_{i_j})\to \be$. Then by \cref{l:converg} either $x\in
\imp_Y(\be)$, or $x\in X\sm Y$ and $\al(x)=\be$. In either case
$m(x)=m(\imp_Y(\be))$ while $m(x_{i_j})=m(\imp_Y(\al(x_{i_j})))$. Since
impressions are upper semi-continuous and $m$ is continuous, then
$m(\imp_Y(\al(x_{i_j})))=m(x_{i_j})\to m(\imp_Y(\be))=m(x)$ as desired.

We claim this implies continuity of $m$. Indeed, suppose that $x_i\to
x$ but $m(x_i)\not \to m(x)$. Refining our sequence we may assume that
$m(x_i)\to t\ne m(x)$. However by the previous paragraph we can find a
subsequence $x_{i_j}$ of $x_i$ such that $m(x_{i_j})\to m(x)$, a
contradiction.

Since for $y\in Y$, $m^{-1}(y)$ is the union of $(m_Y)^{-1}(y)$ and all
components of $X\sm Y$ whose closure intersects  $(m_Y)^{-1}(y)$,
$m^{-1}(y)$ is connected. Hence $m:X\to L$ is the desired monotone map.
\end{proof}

\section{Applications}

In this section we apply our results to complex dynamics.

\subsection{Preliminaries from complex dynamics}
We rely upon basic facts discussed, e.g., in \cite{mil06}. Let us fix a
polynomial $P$ of degree at least two.

\begin{defn}[Periodic points]\label{d:perpo}
A periodic point $p$ of period $n$ is \emph{repelling} if $(P^n)'(x)=
re^{2\pi i\al}$ with $r>1$ and \emph{parabolic} if $(P^n)'(p)=e^{2\pi
i\frac{p}{q}}$, with $p,q\in \mathbb N$. A periodic point $p$ of $P$ of
period $n$ and $(P^n)'(p)=e^{2\pi i\al}$ with $\al\in \mathbb R\sm
\mathbb Q$ is a \emph{Siegel point} if there exists an open disk $U$
containing $p$ so that $P^n|_U$ is analytically conjugate to the rigid
rotation $R(z)= e^{2\pi i\al}z$ of the open unit disk and a
\emph{Cremer point} if such a disk does not exist.
\end{defn}

Periodic points play a  crucial role in complex dynamics; in particular, they are used in
one of the standard equivalent definitions of the \emph{Julia set} of $P$.

\begin{defn}[(Filled) Julia set] \label{d:julia}
The Julia set $J(P)$ of a polynomial $P$ is the closure of the set of
repelling periodic points of $P$; it is known that $J(P)$ is compact.
The set $\C\sm U^\infty_{J(P)}=\toh(J(P))$ is called the \emph{filled Julia set}
and is denoted by $K(P)$.
\end{defn}

The Julia set $J(P)$ coincides with the boundary $\partial
U^\infty_{J(P)}$ of the open set $U^\infty_{J(P)}$ and, hence, $J(P)$
is unshielded. The dynamics of $P$ outside the filled Julia set $K(P)$
is rather predictable.

\begin{defn}[(Non-)escaping points]\label{d:escap}
Points attracted to infinity under iterations of $P$ are called
\emph{escaping}. Otherwise points are said to be \emph{non-escaping}.
\end{defn}

It is known that the unbounded complementary domain $U^\infty_{J(P)}$
of $J(P)$ is in fact the set of all escaping points while its
complement $K(P)$ is in fact the set of all non-escaping points. The
set $U^\infty_{J(P)}=U^\infty_{K(P)}$ is therefore called the \emph{basin of attraction
of infinity}.

The Julia set $J(P)$ is a continuum if and only if all critical points
of $P$ are non-escaping (in other words, the orbits of all critical
points of $P$ are contained in $K(P)$). We will first assume that
$J(P)$ (equivalently, $K(P)$) is connected.  Then it is known that all
repelling and parabolic periodic points of $P$ (and all their
pre-images) are the landing points of finitely many rays
$R_{J(P)}(\al)$ with $\al\in Q$.

In a vast majority of cases the connected Julia set of a polynomial is
either locally connected, or at least admits a non-degenerate finest
locally connected model. However, this is not always the case. To give
an example we need the following alternative definition of a Cremer point.

\begin{defn}\label{d:cremer}
Let $P$ be a polynomial. Suppose that $a$ is a periodic point of $P$ of
period $n$ such that $(P^n)'(a)=e^{2\pi i\theta}$ with $\theta$
irrational. Moreover, suppose that $a$ belongs to the Julia set $J(P)$
of $P$. Then $a$ is  a \emph{Cremer} periodic point of $P$.
\end{defn}

The main result of \cite{bo10} shows that in some cases the finest
locally connected model of a connected Julia set is degenerate.

\begin{thm}[\cite{bo10}]\label{t:cremer}
For the Julia set of a quadratic polynomial with a fixed \emph{Cremer}
point the finest locally connected model is a point.
\end{thm}

In general the existence of a subcontinuum with a non-degenerate finest
locally connected model provides no information about such a model for
the entire unshielded continuum. However, if the subcontinuum is
strategically placed, then Theorem A shows that a non-degenerate model
for the entire space does exist. A natural choice of a subcontinuum of
$J(P)$ on which one can hope to have a non-degenerate finest locally
model is that of a connected Julia set of a \emph{polynomial-like} map
which is a power of $P$. This is another application of polynomial-like
maps that are a powerful tool in complex dynamics introduced by Douady
and Hubbard \cite{dh85}.

\begin{defn}[Polynomial-like maps]\label{d:pl}
A \emph{polynomial-like} map of degree $d$ is a triple $(U, V, f)$
where $U$ and $V$ are open subsets of $\mathbb C$ isomorphic to discs,
with $U$ relatively compact in $V$, and $f: U \to V$ is a proper
analytic map of degree $d$.
\end{defn}

Similar to polynomials, one can define the (filled) Julia set of a
polynomial-like map.

\begin{defn}[(Filled) Julia set of a polynomial-like map]\label{d:plj}
If $f: U\to V$  is a polynomial-like map of degree $d$, we will denote
\[K_f=\bigcap_{n\ge 0}  f^{-n}(U),\]
the compact set of points $z\in U$  such that $f^n(z)$ is defined and
belongs to $U$ for all $n\in \mathbb N$. The set $K_f$ is called the
\emph{filled Julia set} of $f$.  The \emph{Julia set} $J_f$ of $f$ is the
boundary of $K_f$.
\end{defn}

Given a polynomial $P$, we will often say that $P^n|_{K^*}:K^*\to K^*$ (or
$P^n|_{J^*}:J^*\to J^*$) is a polynomial-like map meaning that there
exist open sets $U$ and $V$ as in \cref{d:pl} such that $K^*$ is the
filled Julia set (or $J^*$ is the Julia set) of the corresponding
polynomial-like map $(P^n, U, V)$.

The term \emph{polynomial-like} maps is justified by the Straightening
Theorem stated below. However first we need one more definition.

\begin{defn}[Hybrid equivalence \cite{dh85}]\label{d:hybrid}
Two polynomial-like maps $f:U\to V$ and $g:U'\to V'$ are \emph{hybrid
equivalent} if there is a quasi-conformal map $\vp:U\to U'$ conjugating
$f$ to $g$ such that $\vp$ is conformal almost everywhere on $K(f)$ (in
other words, $\vp$ is such that $\vp \circ f=g\circ \vp$  near $K_f$).
The map $\vp$ is called a \emph{straightening map}.
\end{defn}

An important result of \cite{dh85} is given below; this theorem allows
us to talk about finest locally connected models of connected
\emph{polynomial-like} Julia sets.

\begin{thmS}[\cite{dh85}]
Let $f : U\to V$ be a polynomial-like map. Then $f$ is hybrid
equivalent to a polynomial $P$. Moreover, if $K(f)$ is
connected, then $P$ is unique up to $($global$)$ conjugation by an
affine map.
\end{thmS}

\subsection{Main applications in the connected case}
Suppose that the connected Julia set $J(P)$ of a polynomial $P$
contains a subcontinuum $K^*$ so that $P^n|_{K^*}$ is a polynomial-like
map. Then by the Straightening Theorem $P^n|_{K^*}:K^*\to K^*$ is hybrid
equivalent to a polynomial $g$ with connected filled Julia set $K(g)$.
In particular, under the hybrid equivalence appropriate arcs contained
in external rays of $K(g)$ correspond to arcs inside $U$ which
accumulate to the corresponding polynomial-like Julia set $J^*$ (the
open set $U$ is defined as in \cref{d:pl}). Slightly abusing the
language we will call these arcs \emph{polynomial-like rays} and will
denote them in the same way as we would have denoted external rays of
$K^*$ (or equivalently, of $J^*$), i.e. $R_{K^*}(\al)$ where $\al$ is
the argument of the  external ray of the polynomial
$g$ corresponding to $R_{K^*}(\al)$.

Recall that an external ray $R_J(\al)$ is said to
\emph{accumulate} in $J^*$ if $\Pr(\al)\subset J^*$. Also, it is easy
to see that the property of a point being repelling or parabolic is
preserved under hybrid equivalence. By \cref{d:julia} this allows one
to conclude that repelling periodic points of $P$ are dense in $J^*$.
Moreover, it follows that if $p$ is a repelling or parabolic periodic
point of $P$, then only finitely many external rays $R_{J^*}(\al)$ of $J^*$ and
finitely many external rays $R_J(\be)$ of $J$ land on $p$.


Suppose that $Y\subset X$ are unshielded plane continua.
Above in \cref{homotope}
we considered a map $p:\mathcal A\to\uc$; this map associated to a ray
$R_Y(\al)$ the ray $R_X(p(\al))$ so that both rays landed on the same
point $y_\al\in Y$ and were homotopic outside $Y$ by a homotopy fixing $y_\al$. In the case of
polynomials $f$ and polynomial-like maps $f^*$ it is easier to first consider the
''inverse''  map which associates rays $R_{J(f)}(\beta)$ which land on a
point $y_\beta\in J(f^*)$ to rays $R_{J(f^*)}(\nu(\beta))$ which land on $y_\be$ and are homotopic
to $R_X(\be)$ outside $Y$ by a homotopy which fixes $y_\be$.
This is accomplished in \cref{l:weakmonotone}.

In what follows, given a map, we call a point \emph{preperiodic} if it
is not periodic but eventually maps to a periodic point, and
\emph{(pre)periodic} if it is periodic or preperiodic. Recall that if
the Julia set $J(P)$ of a polynomial $P$ is connected and an angle
$\al$ is (pre)periodic then the external ray $R_{J(P)}$ lands on a
(pre)periodic (in the sense of $P$) point in $J(P)$ \cite{mil06}. Given
a set $T\subset \uc$ we say that a map $\Psi:T\to \uc$ is
\emph{extendably monotone} if
$\Psi$ has a monotone (but not necessarily continuous!) extension $m:\uc\to\uc$.

\begin{lem}\label{l:weakmonotone}
Suppose that $P$ is a polynomial of degree $d$ with connected Julia set
$J$ and $J^*\subset K$ is a subcontinuum of $J$ such that $P^n|_{J^*}$
is a polynomial-like map with filled Julia set $K^*$ and Julia set
$J^*$. Suppose that $P^n|_{J^*}$ is hybrid equivalent to a polynomial
$Q$ of degree $k$. Let $\mathcal B\subset \uc$ be the set of
\textbf{all} angles $\beta$ so that $R_J(\beta)$ lands on a point
$y_\beta\in J^*$.  Then there exists a extendably monotone continuous
map $\nu:\mathcal B\to\uc$ such that:
\begin{enumerate}
\item for each $\beta\in \mathcal B$ the ray $R_{J^*}(\nu(\beta))$ lands on the same point
$y_\beta$ and the rays $R_J(\beta)$ and $R_{J^*}(\nu(\beta))$ are homotopic outside $K^*$ under a homotopy which fixes the point $y_\beta$,
\item if $\mathcal B'\subset \mathcal B$ is the set of all (pre)periodic angles, then $\nu(\mathcal B')$ is dense in $\uc$,
\item $\nu\circ \si_d=\si_k\circ \nu$.
\end{enumerate}
\end{lem}

Notice that the continuity of $\nu$ on $\mathcal B$ only means that
$\nu$ is continuous at points of $\mathcal B$ and does not imply that
$\nu$ can be extended to a continuous monotone map of the circle to
itself.

\begin{proof}
Since $P^n|_{J^*}$ is polynomial-like, there exist Jordan disks
$U\subset \ol{U}\subset V$ such that $J^*\subset U$ and $P^n:U\to V$ is
polynomial-like. Denote $P^n|_U$ by $P^*$.

Let $R_J(\beta)$ be an external ray of $J$ which lands on a 
point $y_\be\in J^*$. 
Consider the inverse $\xi: U^\infty_{J^*}\to \disk^\infty$ of the
corresponding Riemann map from $\disk^\infty$ to $ U^\infty_{J^*}$ with
derivative converging to a real number at infinity. Then
$\xi(R_J(\be))$ is a curve which accumulates at a point $z\in \uc$.
Choose the polynomial-like ray $R_{J^*}(\al)$ of $J^*$ whose
$\xi$-image is the radial ray to $\disk^\infty$ landing at $z$ (the
argument of this radial ray and hence the argument of the corresponding
polynomial-like ray is denoted by $\al$). Since in the
$\disk^\infty$-plane the radial ray to $z$ and $\xi(R_J(\al))$ are
homotopic, it follows that $R_J(\be)$ and $R_{J^*}(\al)$ are homotopic
outside $J^*$ by a homotopy which fixes $y$ (the homotopy carries over
to $\C\sm J^*$ under the Riemann map). Define $\nu(\be)=\al$. Since
this construction goes through for all angles $\be\in \mathcal B$, this
defines a map $\nu:\mathcal B\to \uc$.

To see that $\nu$ is extendably monotone suppose that
$\nu(\be_1)=\nu(\be_2)$. Then $\xi(R_J(\be_1))$ and $\xi(R_J(\be_2))$
are two curves which land on the same point $z\in \uc$. Denote by $T$
the component of $\disk^\infty \sm [\xi(R_J(\be_1))
\cup\xi(R_J(\be_2))]$  whose closure meets $\uc$ only in the point
$z\in \uc$ (in other words, $T$ is the wedge between $\xi(R_J(\be_1))$
and $\xi(R_K(\be_2))$ which does not contain the unit disk). Then any
external ray $R_J(\ga)$ with $\xi(R_J(\ga))\subset T$ that lands on a
point of $J^*$ must land on $y$ so that $\xi(R_J(\ga))$ lands on $z$.

This implies that there exists an arc $A_z\subset \uc$ so that
$\nu^{-1}(z)=A_z\cap \mathcal B$. To see that there exists a monotone
extension of $\nu$ it remains to observe that circular orientation
among points of $\mathcal B$ is preserved under $\nu$ in the following
sense: if $\be_1<\be_1<\be_3$ then it is impossible that
$\nu(\be_1)<\nu(\be_3)<\nu(\be_2)$ as otherwise some external rays of
$K$ will have to intersect. Thus, the arcs $A_z$ constructed above for
all points $z\in \nu(\mathcal B)$ have the same circular order as the
points $z\in \uc$ which implies the desired.

Now, choose an angle $\be\in \mathcal B$ such that $y_\be\in J^*$, the
landing point of the external ray $R_J(\be)$, is preperiodic. Set
$\al=\nu(\be)$. Properties of polynomials (and hence of polynomial-like
maps) imply that the family of all polynomial-like rays which are
preimages of $R_{J^*}(\al)$ is such that their arguments are dense in
$\uc$. Each such polynomial-like ray $R^*$ with argument $\al'$ is a
unique pullback of $R_{J^*}(\al)$ under the appropriate branch of the
inverse function to $P^*$ (recall that $y_\be$ is not periodic). If we
simultaneously pull back $R_J(\be)$ under the same branch of the
inverse function of $P^*$ we will obtain an external ray $R_J(\be')$ of $J$
with argument $\be'$ which lands on the same point as $R^*$ and is
homotopic to $R^*$ outside $K^*$.  Denote the argument of $R$ by
$\al'$, then $\nu(\be')=\al'$. This shows that (2) holds.

To see that $\nu$ is continuous consider a sequence $\be_1<\be_2<\dots$
in $\mathcal B$ so that $\lim \be_i=\be_\infty\in\mathcal B$. Consider
the landing points $z_i$ of the curves $\xi(R_J(\be_i))$ and the
landing point $z_\infty$ of $\xi(R_J(\be_\infty))$. The fact that $\nu$
is extendably monotone implies that $z_1\le z_2\le \dots \le z_\infty$.
We claim that $z_\infty=\lim z_i$. Indeed, otherwise we have that
$z_1\le \lim z_i=t<z_\infty$. By (2) we can choose a (pre)periodic
angle $\hat \be\in \mathcal B'$ such that $t<\nu(\hat\be)<z_\infty$.
Since $\nu$ is extendably monotone this contradicts the fact that $\lim
\be_i=\be_\infty$. Thus, $z_\infty=\lim z_i$ as desired. The last claim
of the lemma is left to the reader.
\end{proof}

The following corollary easily follows from definitions, \cref{homotope}
and \cref{l:weakmonotone}

\begin{cor}\label{l:pl-ray-corr}
Suppose that the connected Julia set $J(P)$ of a polynomial $P$
contains a subcontinuum $J^*$ so that $P^n|_{J^*}$ is a polynomial-like
map for some $n\ge 1$. Then $J^*$ is strategically placed in $J(P)$.
\end{cor}


\begin{proof}
Let us use the notation from \cref{l:weakmonotone}. Set $\Ac=\nu(\mathcal B)$. Then by
\cref{l:weakmonotone} the set $\Ac$ is dense in $\uc$. Moreover, by \cref{l:weakmonotone}
conditions listed in \cref{homotope}(2) are satisfied for $\Ac\subset \uc$
and $J^*\subset J(P)$. Hence $J^*$ is strategically placed in $J(P)$.
\end{proof}

\cref{l:pl-ray-corr} allows one to conclude that connected
polynomial-like Julia sets with non-degenerate finest locally connected
models force the existence of non-degenerate finest locally connected
models of containing them connected polynomial Julia sets.

\begin{thmb}
Suppose that $P:\C\to \C$ is a polynomial and  $J^*\subset J(P)$  is a
continuum which is a polynomial-like Julia set of $P^n$ for some $n>0$.
If $J^*$ has a non-degenerate finest locally connected model, then so
does $J(P)$.
\end{thmb}

\begin{proof}
Indeed, by \cref{l:pl-ray-corr} Theorem A implies the desired.
\end{proof}

Note that if $K^*$ is a filled polynomial-like Julia set of a polynomial $P$,
then $K^*$ is a component of $P^{-n}(K^*)$. As it turns out this is
almost sufficient (the proof of \cref{t:poly-like} uses some ideas
communicated by M. Lyubich to the third named author). For convenience we state
these results in the case that $n=1$.

\begin{thm}[Theorem B \cite{bopt15}]\label{t:poly-like}
  Let $P:\C\to\C$ be a polynomial, and $Y\subset\C$ be a full
  $P$-invariant continuum. The following assertions are equivalent:
\begin{enumerate}
 \item the set $Y$ is the filled Julia set of some polynomial-like
     map $P:U^*\to V^*$ of degree $k$,
 \item $Y$ is a component of the set $P^{-1}(P(Y))$, and, for every
     attracting or parabolic point $y$ of $P$ in $Y$, the immediate
     attracting basin of $y$ or the union of all parabolic domains
     at $y$ is a subset of $Y$.
\end{enumerate}
\end{thm}

The following corollary is now almost immediate.

\begin{cor}
Suppose that $K^*\subset K$ is a subcontinuum of the filled Julia set
of a polynomial $P$ such that $K^*$ is a component of $P^{-1}\circ P
(K^*)$ containing all immediate parabolic and attracting basins of all
attracting and parabolic points in $K^*$. Then if
$\partial(U^\infty_{K^*})$ has a non-degenerate finest locally
connected model, then $J(P)$ has a non-degenerate finest locally
connected model.
\end{cor}

\begin{proof}
By \cite{bopt15}, $P|_{K^*}:K^*\to K^*$ is a polynomial-like map. Hence
the result follows from Theorem B.
\end{proof}

\subsection{Models for non-connected spaces}\label{ss:non-conn}

Models for non-connected spaces  were studied in \cite{bco13}. A
\emph{compactum} is a compact metric space. Since a compactum with
infinitely many distinct components  is always not locally connected at
some point, we need to replace the condition of local connectedness of
the model by a suitable notion.

A compactum X is called \emph{finitely Suslinian} if, for every $\e >
0$, every collection of disjoint subcontinua of X with diameters at
least $\e$ is finite.  By Lemma 2.9 [BO04], unshielded planar locally
connected continua are finitely Suslinian and vice versa. Thus, in the
unshielded case,  the notion of finitely Suslinian generalizes the
notion of local connectivity. This motivates us to look for good
finitely Suslinian models of planar compacta.

\begin{defn}
Let $X$  be a compactum. A \emph{finest finitely Suslinian model} for
$X$ is a finitely Suslinian compactum $S$ and a monotone map $m:X\to S$
so that for each monotone map $f:X\to Y$ to a finitely Suslinian
compactum $Y$ there exists a monotone map $g:S\to Y$ with $g\circ m=f$.
Then the map $m:X\to S$ is called a \emph{finest finitely Suslinian
model map}. We say that a compactum $X$ has a \emph{non-degenerate
finitely Suslinian model} $S$ if at least one component of $S$ is
non-degenerate.
\end{defn}

Observe that by definition of a monotone map it follows that if $m$ is
monotone then distinct components of $X$ map to distinct components of
$m(X)$. Observe also that the above introduced notion of a degenerate
finitely Suslinian model agrees with the notion of a degenerate locally
connected model in the case of continua.

By \cite{bco13} all finest finitely Suslinian models of a compactum $X$
are homeomorphic and we can talk about \emph{the} finest finitely
Suslinian models of compacta. It was shown in \cite{bco13} that every
planar unshielded compactum $X$ has a  finest finitely Suslinian model
$S$ (which is unique up to homeomorphisms). As previously in the case
of continua, the finest finitely
Suslinian model $S$ of $X$ may be degenerate (i.e., the finest finitely
Suslinian model monotone map $m:X\to S$ may well collapse \emph{all}
components of $X$ to points). The following theorem is the main result
of \cite{bco13} concerning finest finitely Suslinian models of
polynomial Julia sets (this time including disconnected Julia sets).


\begin{thm}[Theorem 6 \cite{bco13}]\label{t:fin-sus}
The finest finitely Sus\-li\-ni\-an mo\-del monotone map $m:
J(P)\to S$ of the Julia set $J(P)$ of a polynomial $P$ coincides on
each component $X$ of $J(P)$ with the finest monotone map $m_X$ of $X$
to a locally connected continuum. In particular, the following holds:

\begin{enumerate}

\item the finest finitely Suslinian monotone model of $J(P)$ is
    non-degenerate if and only if there exists a periodic component
    of $J(P)$ whose finest finitely Suslinian monotone model is
    non-degenerate;

\item the Julia set $J(P)$ is finitely Suslinian if and only if all
    periodic non-degenerate components of $J(P)$ are locally
    connected.

\end{enumerate}

\end{thm}

Hence, the following theorem immediately follows.

\begin{thm}
Suppose that $J$ is the Julia set of a polynomial $P$ and
$J^*\subset J$ is a subcontinuum so that, for some integer $r$,
$P^r|_{J^*}:J^*\to J^*$ is a polynomial-like map and $J^*$ has a
non-degenerate finest locally connected model. Then $J$ has the finest
finitely Suslinian model.
\end{thm}

\begin{proof}
Suppose that $K^*$ is contained in the component $C$ of $J$. Then $C$
must be periodic of some period $n$. By a result of \cite{bh88},
$P^n|_C:C\to C$ is a polynomial-like map. Hence $P^n|_C$ is hybrid
equivalent to a polynomial $g$. Since $J^*\subset C$ it follows from
Theorem~B that $C$ has a non-degenerate finest locally connected model.
Hence,  by Theorem~\ref{t:fin-sus}, $J$ has the finest finitely
Suslinian model.
\end{proof}

\bibliographystyle{amsalpha}
\providecommand{\bysame}{\leavevmode\hbox to3em{\hrulefill}\thinspace}
\providecommand{\MR}{\relax\ifhmode\unskip\space\fi MR }
\providecommand{\MRhref}[2]{%
  \href{http://www.ams.org/mathscinet-getitem?mr=#1}{#2}
} \providecommand{\href}[2]{#2}

\end{document}